\documentclass[12pt]{article}
\usepackage{amssymb,amsthm,amsmath,latexsym}
\usepackage{url}
\newtheorem{thm}{Theorem}
\newtheorem{prop}[thm]{Proposition}

\theoremstyle{remark}

\newcommand{\FF}{\mathbb{F}}

\DeclareMathOperator{\wt}{wt}

\DeclareMathOperator{\supp}{supp}

\begin{document}
\title{Singly even self-dual codes of length $24k+10$ and
minimum weight $4k+2$}

\author{
Masaaki Harada\thanks{
Research Center for Pure and Applied Mathematics,
Graduate School of Information Sciences,
Tohoku University, Sendai 980--8579, Japan.
email: mharada@m.tohoku.ac.jp.}
}

\maketitle

\begin{abstract}
Currently, the existence of an extremal singly even self-dual 
code of length $24k+10$ is unknown for all nonnegative integers $k$.
In this note, 
we study singly even self-dual $[24k+10,12k+5,4k+2]$ codes.
We give some restrictions on the possible
weight enumerators of singly even self-dual $[24k+10,12k+5,4k+2]$
codes with shadows of minimum weight at least $5$
for $k=2,3,4,5$.
We discuss a method for constructing singly even self-dual
codes with minimal shadow.
As an example, a singly even self-dual $[82,41,14]$ code
with minimal shadow is constructed for the first time.
In addition, as neighbors of the code, 
we construct singly even self-dual $[82,41,14]$ codes 
with weight enumerator for which no singly even
self-dual code was previously known to exist.
\end{abstract}

%

\section{Introduction}

Extremal self-dual codes are an important class of linear codes 
for both theoretical and practical reasons.
It is a fundamental problem to determine the largest
minimum weight among self-dual codes of that length,
and much work has been done concerning this problem.

A (binary) code $C$ of length $n$ is a vector subspace of
$\FF_2^n$, where $\FF_2$ denotes the finite field of order $2$.
All codes in this note are binary.
The {\em dual} code $C^{\perp}$ of $C$ is defined as
$
C^{\perp}=
\{x \in \FF_2^n \mid x \cdot y = 0 \text{ for all } y \in C\},
$
where $x \cdot y$ is the standard inner product.
A code $C$ is called
{\em self-dual} if $C = C^{\perp}$.
Self-dual codes are divided into two classes.
A self-dual code $C$ is {\em doubly even} if all
codewords $x$ of $C$ have weight $\wt(x) \equiv 0 \pmod 4$, and {\em
singly even} if there is at least one codeword of weight $\equiv 2
\pmod 4$.
A doubly even self-dual code of length $n$ exists if and
only if $n \equiv 0 \pmod 8$, while a singly even self-dual code
of length $n$ exists if and
only if $n$ is even.

Let $C$ be a singly even self-dual code.
Let $C_0$ denote the 
subcode of $C$ consisting of 
codewords $x$ having weight $\wt(x) \equiv0\pmod4$.
The {\em shadow} $S$ of $C$ is defined to be $C_0^\perp \setminus C$.
A singly even self-dual code of length $n$
is called a code with {\em minimal shadow} 
if the minimum weight of the shadow is $4,1,2$ and $3$ 
if $n \equiv 0,2,4$ and $6 \pmod 8$,
respectively.
The concept of singly even self-dual codes with minimal shadow 
was introduced in~\cite{BMW}.

Rains~\cite{Rains} showed that
the minimum weight $d$ of a self-dual code of length
$n$ is bounded by $d\le 4 \lfloor n/24 \rfloor +4$ unless 
$n \equiv 22 \pmod{24}$ when $d \le 4 \lfloor n/24
\rfloor+6$.
A self-dual code meeting the upper bound is called {\em extremal}.
We say that a self-dual code is {\em optimal} if it has
the largest minimum weight among all self-dual codes of
that length.
For length $24k+10$ $(k=0,1,\ldots,5)$,
we give the current information on the largest minimum
weight $d(24k+10)$:
\begin{multline*}
d(10)=2,
d(34)=6,
d(58)=10,
d(82)=14 \text{ or }16,
\\
d(106)=16 \text{ or }18,
d(130)=20, 22 \text{ or }24,
\end{multline*}
(see~\cite[Table~I]{C-S}, \cite[Table~VI]{DGH}, \cite[Table~I]{HKWY}).
Currently, the existence of an extremal singly even self-dual 
code of length $24k+10$ is unknown for all nonnegative integers $k$.
In addition, Han and Lee~\cite{HL} conjecture that
there is no extremal singly even self-dual code of length $24k+10$
for all nonnegative integers $k$.
It was shown in~\cite{BW} that there is no extremal singly even
self-dual code with minimal shadow
for length $24k+10$.
These motivate our interest in singly even self-dual
$[24k+10,12k+5,4k+2]$ codes.

This note is organized as follows.
In Section~\ref{sec:WE},
the possible weight enumerators of singly even self-dual
$[82,41,14]$ codes are determined.
In addition, in Section~\ref{sec:res},
we give some restrictions on the possible
weight enumerators of singly even self-dual $[24k+10,12k+5,4k+2]$
codes with shadows of minimum weight at least $5$
for $k=2,3,4,5$.
In Section~\ref{sec:const},
we discuss a method for constructing singly even self-dual
codes with minimal shadow.
As an example, a singly even self-dual $[82,41,14]$ code $C_{82}$
with minimal shadow is constructed for the first time.
Finally, 
in Section~\ref{sec:82n}, as neighbors of $C_{82}$, 
we construct singly even self-dual $[82,41,14]$ codes 
with weight enumerator for which no singly even
self-dual code was previously known to exist.
It is a fundamental problem to find which weight enumerators 
actually occur for the possible weight enumerators.
We emphasize that singly even self-dual $[82,41,14]$ codes with shadows
of minimum weight $1,5,9$ are constructed for the first time.

All computer calculations in this note
were done with the help of 
the algebra software {\sc Magma}~\cite{Magma} and
the mathematical software {\sc Mathematica}.

\section{Weight enumerators of singly even self-dual $[82,41,14]$ codes}
\label{sec:WE}

Let $C$ be a singly even self-dual code of
length $n$ with shadow $S$.
Let $A_i$ and $B_i$ be the numbers of vectors of
weight $i$ in $C$ and $S$, respectively.
The weight enumerators $W_C$ and $W_S$ of $C$ and $S$ are given by
$\sum_{i=0}^n A_i y^i$
and
$\sum_{i=d(S)}^{n-d(S)} B_i y^i$, respectively,
where $d(S)$ denotes the minimum weight of $S$.
If we write
\[
W_C = \sum_{j=0}^{ \lfloor n/8  \rfloor}
a_j(1+y^2)^{n/2-4j}(y^2(1-y^2)^2)^j,
\]
for suitable integers $a_j$, then
\[
W_S  = \sum_{j=0}^{ \lfloor n/8
\rfloor}(-1)^j a_j2^{n/2-6j}y^{n/2-4j}(1-y^4)^{2j},
\]
\cite[(10), (11)]{C-S}.
Suppose that $C$ is a singly even self-dual $[82,41,14]$ code.
Since the minimum weight is $14$, we have
\begin{multline*}
a_0=1,
a_1=-41,
a_2=615,
a_3=-4182,
\\
a_4=13161,
a_5=-18040,
a_6=9512.
\end{multline*}
Then the weight enumerator of the shadow $S$ is written as:
\begin{align*}
&
\frac{a_{10}}{524288} y
+\left(- \frac{a_9}{8192} - \frac{5 a_{10}}{131072} \right) y^5
+\left(\frac{a_8}{128} + \frac{9a_9}{4096} +\frac{95a_{10}}{262144} 
\right) y^9
\\&
+\left(- \frac{a_7}{2} -  \frac{a_8}{8} - 
\frac{153a_9}{8192} -\frac{285a_{10}}{131072} \right) y^{13}
+ \cdots.
\end{align*}

\begin{itemize}
\item $d(S)=1$:
From~\cite[(6)]{C-S}, 
$S$ has a unique vector of weight $1$ and $S$
has no vector of weights $5$ and $9$.
Hence, $a_{10}=524288$, $a_9=-163840$ and $a_8=21760$.
Since $A_{14}=B_{13}$ by~\cite{H60}, 
\[
3280 + a_7 = - 800 - \frac{a_7}{2}. 
\]
Thus, we have that $a_7=-2720$.
Therefore, we have the following possible weight enumerators
\begin{align*}
W^C_{82,1}=&
1 + 560 y^{14} + 60724 y^{16} + 233545 y^{18}
+ \cdots,
\\
W^S_{82,1}=&
y + 560 y^{13} + 294269 y^{17} + 33367568 y^{21} + \cdots,
\end{align*} 
respectively.

\item $d(S)=5$:
From~\cite[(6)]{C-S}, we have $a_{10}=0$ and $a_9=-8192$.
Then we have that
$a_8$ is divisible by $128$, say $a_8=128\alpha$ and
$a_7$ is divisible by $2$, say $a_7=2\beta$, 
where $\alpha$ and $\beta$ are integers.
Therefore, we have the following possible weight enumerators
\begin{align*}
W^C_{82,2}=&
1
+( 3280 + 2 \beta )y^{14}
+( 36244  + 128 \alpha - 2 \beta )y^{16}
\\&
+( 506153  - 896 \alpha - 26 \beta )y^{18}
+ \cdots,
\\
W^S_{82,2}=&
y^5
+(- 18 + \alpha) y^9
+( 153 - 16 \alpha - \beta) y^{13}
\\&
+( 303568  + 120 \alpha  + 14 \beta) y^{17}
+ \cdots,
\end{align*} 
respectively.
Note that $\alpha \ge 18$ and $\beta \le 153-16\alpha \le -135$.
In Section~\ref{sec:res},
it is shown that $\beta$ is an even integer 
(see Proposition~\ref{prop:W82}).

\item $d(S) \ge 9$:
Then we have $a_{10}=a_9=0$.
We have that
$a_8$ is divisible by $128$, say $a_8=128\alpha$ and
$a_7$ is divisible by $2$, say $a_7=2\beta$, 
where $\alpha$ and $\beta$ are integers.
Therefore, we have the following possible weight enumerators
\begin{align*}
W^C_{82,3}=&
1
+( 3280  + 2 \beta )y^{14}
+( 36244  + 128 \alpha  - 2 \beta )y^{16}
\\&
+( 514345  - 896 \alpha - 26 \beta )y^{18}
+ \cdots,
\\
W^S_{82,3}=&
\alpha y^9
+( - 16 \alpha - \beta )y^{13}
+( 304384 + 120 \alpha + 14 \beta )y^{17}
\\&
+( 33293312 - 560 \alpha - 91 \beta )y^{21}
+ \cdots,
\end{align*} 
respectively.  Note that
$\alpha \ge 0$ and $\beta \le -16\alpha \le 0$.
In Section~\ref{sec:res},
it is shown that $\beta$ is an even integer 
(see Proposition~\ref{prop:W82}).
\end{itemize}

It is unknown whether there is a singly even 
self-dual $[82,41,16]$ code.
The first example of a singly even 
self-dual $[82,41,14]$ code was found
in~\cite[Section V]{DGH}.
The weight enumerators of the code and its shadow were given,
however unfortunately the weight enumerator of the shadow
was incorrectly stated and the correct weight enumerator is
\[
656 y^{13} + 295200 y^{17} + 33353008 y^{21} + \cdots.
\]
This code has weight enumerator $W^C_{82,3}$ with 
\begin{equation}\label{eq:82}
(\alpha,\beta)=(0,-656). 
\end{equation}
The code was the only  previously known singly
even self-dual $[82, 41, 14]$ code.
In Sections~\ref{sec:const} and~\ref{sec:82n}, 
we construct singly even self-dual $[82,41,14]$ codes 
with weight enumerator for which no singly even
self-dual code was previously known to exist.

\section{Restrictions on weight enumerators of singly even
self-dual $[24k+10,12k+5,4k+2]$ codes}
\label{sec:res}


It was shown in~\cite{BHM} that
the weight enumerator of a singly even self-dual 
$[24k+10,12k+5,4k+2]$ code with minimal shadow is uniquely 
determined.
In this section, we give some restrictions on the possible
weight enumerators of singly even self-dual $[24k+10,12k+5,4k+2]$
codes with shadows of minimum weight at least $5$
for $k=2,3,4,5$.
It is a key idea to consider the possible weight enumerator
of $C_1$.

\subsection{Possible weight enumerators of $C_1$}

Let $C$ be a singly even self-dual $[24k+10,12k+5,4k+2]$ code.
Let $W^{(1)}$ and $W^{(3)}$ denote the weight enumerators
of $C_1$ and $C_3$, respectively.
By~\cite[Theorem~5, 5)]{C-S}, the possible weight enumerators
$W^{(1)}-W^{(3)}$ are written as:
\begin{equation}\label{eq:W1W3}
W^{(1)}-W^{(3)}
=
\sum_{i=0}^{k-1} b_i
(1 + 14 y^4 + y^8)^{3k-1-3i}
(y^4 (1 - y^4)^4)^i
f(y)
\end{equation}
where 
$f(y)=y - 34 y^5 + 34 y^{13} - y^{17}$
and $b_0,b_1,\ldots,b_{k-1}$ are integers.
Combined with the possible weight enumerators of the shadow,
using~\eqref{eq:W1W3}, the possible weight enumerators of $C_1$
are determined.  

\subsection{Optimal singly even self-dual $[58,29,10]$ codes 
with shadows of minimum weight at least $5$}

The possible weight enumerators of optimal singly even
self-dual $[58,29,10]$ codes with shadow of minimum
weight at least $5$ and the shadows
are known as follows:
\begin{align*}
W^C_{58} =& 1 +(319 - 24 \beta - 2 \gamma)y^{10} +
(3132 + 152 \beta + 2\gamma)y^{12} + \cdots, 
\\
W^S_{58} =&
\beta y^5
+ \gamma y^9
+ (24128 - 54 \beta - 10 \gamma )y^{13}
\\&
+ (1469952 + 320 \beta + 45 \gamma )y^{17}
+ \cdots,
\end{align*}
respectively, where $\beta, \gamma$ are integers~\cite{C-S}.
If there is an optimal singly even self-dual
$[58,29,10]$ code with weight enumerator $W^C_{58}$,
then $\beta \in \{0,1,2\}$~\cite{HM}.
An optimal singly even self-dual code with weight enumerator
$W^C_{58}$ is known for
\begin{align*}
\beta=0 \text{ and }
& \gamma\in \{2m \mid m =0,1,\ldots,65,68,71,79\}, \\
\beta=1  \text{ and }
& \gamma\in \{2m \mid m =8,9,\ldots,58,63\}, \\
\beta=2  \text{ and }
& \gamma\in \{2m \mid m =0,4,6,\ldots,55\}
\end{align*}
(see~\cite{H60}).

\begin{thm}\label{thm:W58}
If there is an optimal singly even self-dual $[58,29,10]$
code with weight enumerator $W^C_{58}$, then
$\gamma$ must be an even integer.
\end{thm}
\begin{proof}
Let $C$ be an optimal singly even self-dual $[58,29,10]$
code with weight enumerator $W^C_{58}$.
From \eqref{eq:W1W3}, we obtain
\begin{align*}
W^{(1)}-W^{(3)}=&
b_0 y
+( 36 b_0 + b_1 )y^5
+(- 415 b_0 - 10 b_1 )y^9
\\&
+(- 39056 b_0 -  724 b_1 )y^{13}
+(- 742131 b_0  - 3694 b_1 )y^{17}
+ \cdots.
\end{align*}
Since the shadow contains no vector of weight $1$,
we have that $b_0=0$.
Hence, we have
\begin{align*}
W^{(1)}=&
\frac{1}{2}(b_1+ \beta) y^5
+\left(- 5 b_1  + \frac{\gamma}{2}\right) y^9
+( 12064  - 362 b_1  - 27 \beta  - 5 \gamma )y^{13}
\\&
+\left( 734976  - 1847 b_1  + 160 \beta  + \frac{45 \gamma}{2}\right)y^{17}
+ \cdots.
\end{align*}
The result follows.
\end{proof}

\subsection{Singly even self-dual $[82,41,14]$ codes
with shadows of minimum weight at least $5$}

\begin{prop}\label{prop:W82}
If there is a singly even self-dual $[82,41,14]$
code with weight enumerator $W^C_{82,i}$, then
$\beta$ must be an even integer for $i=2,3$.
\end{prop}
\begin{proof}
Let $C$ be a singly even self-dual $[82,41,14]$
code with shadow $S$ of minimum weight at least $5$.
From $W^C_{82,i}$ and $W^S_{82,i}$ ($i=2,3$),
the possible weight enumerators of $C$ and $S$
are written using integers $a,b,c$
\begin{align*}
W^C_{82}=&
1
+( 3280 + 2 c )y^{14}
+( 36244 + 128 b - 2 c )y^{16}
\\&
+( 514345 - 8192 a - 896 b - 26 c )y^{18}
+ \cdots,
\\
W^S_{82}=&
a y^5
+(- 18 a  + b) y^9
+( 153 a - 16 b - c) y^{13}
\\&
+( 304384 - 816 a + 120 b + 14 c) y^{17}
+ \cdots,
\end{align*}
respectively.
From \eqref{eq:W1W3}, we obtain
\begin{align*}
W^{(1)}-W^{(3)}=&
b_0 y
+( 78 b_0  + b_1 )y^5
+( 1688 b_0  + 32 b_1  + b_2 )y^9
\\&
+(- 32382 b_0  - 553 b_1  - 14 b_2 )y^{13}
\\&
+(- 2525349 b_0  - 37184 b_1  - 678 b_2 )y^{17}
+ \cdots.
\end{align*}
Since the shadow contains no vector of weight $1$,
we have that $b_0=0$.
Hence, we have
\begin{align*}
W^{(1)}=&
\frac{a + b_1}{2}y^5
+\left(- 9 a + 16 b_1 + \frac{b_2+ b}{2}\right)y^9
\\&
+\left(\frac{153a- 553 b_1-c}{2} - 7 b_2 - 8 b\right)y^{13}
\\&
+(152192 - 408 a - 18592 b_1 - 339 b_2 + 60 b + 7 c) y^{17}
+ \cdots.
\end{align*}
Since $a+b_1$ is even, 
$c$ must be an even integer.
The result follows.
\end{proof}

\subsection{Singly even self-dual $[106,53,18]$ codes
with shadows of minimum weight at least $5$}

By a method  similar to that given in Section~\ref{sec:WE},
the possible weight 
enumerators of singly even self-dual $[106,53,18]$ codes
with shadows of minimum weight at least $5$ 
and the shadows are determined as follows:
\begin{align*}
W^C_{106}=&
1
+(35245 + 2 d )y^{18}
+(416262 + 128 c - 2 d )y^{20}
\\&
+(6586310 + 8192 b - 896 c - 34 d )y^{22}
\\&
+(86626645 + 524288 a - 106496 b + 1024 c + 34 d )y^{24}
+ \cdots,
\\
W^S_{106}=&
a y^5
+(- 24 a - b )y^9
+( 276 a + 22 b + c )y^{13}
\\&
+(- 2024 a - 231 b - 20 c - d )y^{17}
+ \cdots,
\end{align*}
respectively, where $a,b,c,d$ are integers.
If $a=0$, then $-b \in \{0,1,2\}$ by~\cite[Lemma~2]{HM}.

\begin{prop}\label{prop:W106}
If there is a singly even self-dual $[106,53,18]$
code with weight enumerator $W^C_{106}$, then
$d$ must be an even integer.
\end{prop}
\begin{proof}
Let $C$ be a singly even self-dual $[106,53,18]$
code with weight enumerator $W^C_{106}$.
From \eqref{eq:W1W3}, we obtain
\begin{align*}
W^{(1)}-W^{(3)}=&
b_0 y
+( 120 b_0 + b_1 )y^5
+( 5555 b_0 + 74 b_1 + b_2 )y^9
\\&
+( 87440 b_0 + 1382 b_1 + 28 b_2 + b_3 )y^{13}
\\&
+(- 2666610 b_0 - 38670 b_1 - 675 b_2 - 18 b_3 )y^{17}
+ \cdots.
\end{align*}
Since the shadow has minimum weight at least $5$,
we have $b_0=0$.
Hence, we have
\begin{align*}
W^{(1)}=&
\frac{a +b_1}{2} y^5
+\left(- 12 a  + 37 b_1 + \frac{b_2 - b}{2}\right)y^9
\\&
+\left(138 a + 691 b_1 + 14 b_2+ 11 b + \frac{b_3 +c}{2}\right)y^{13}
\\&
+\left(- 1012 a - 19335 b_1 - 9 b_3 - 10 c - \frac{675 b_2 + 231 b 
+ d}{2}\right)y^{17}
+ \cdots.
\end{align*}
Since $b_2-b$ is even, $d$ must be even.
The result follows.
\end{proof}

It is unknown whether there is a singly even self-dual
$[106,53,18]$ code or not (see~\cite[Table~I]{HKWY}).

\subsection{Singly even self-dual $[130,65,22]$ codes
with shadows of minimum weight at least $5$}

By a method similar to that given in Section~\ref{sec:WE},
the possible weight 
enumerators of singly even self-dual $[130,65,22]$ codes
with shadows of minimum weight at least $5$ 
and the shadows are determined as follows:
\begin{align*}
W^C_{130}=&
1
+(388700 + 2 e)y^{22}
+(4791150 + 128 d - 2 e)y^{24}
\\&
+(81082890 + 8192 c - 896 d - 42 e)y^{26}
\\&
+(1200197180 + 524288 b - 106496 c + 512 d + 42 e)y^{28}
\\&
+(14196225992 - 33554432 a - 9961472 b + 532480 c 
\\& 
\qquad + 10752 d +  420 e) y^{30}
+ \cdots,
\\
W^S_{130}=&
a y^5
+(- 30 a + b )y^9
+(435 a - 28 b  - c )y^{13}
\\&
+(- 4060 a + 378 b + 26 c + d )y^{17}
\\&
+(27405 a - 3276 b - 325 c - 24 d - e )y^{21}
+ \cdots,
\end{align*}
respectively, where $a,b,c,d,e$ are integers.

\begin{prop}\label{prop:W130}
If there is a singly even self-dual $[130,65,22]$
code with weight enumerator $W^C_{130}$, then
$e$ must be an even integer.
\end{prop}
\begin{proof}
Let $C$ be a singly even self-dual $[130,65,22]$
code with weight enumerator $W^C_{130}$.
From \eqref{eq:W1W3}, we obtain
\begin{align*}
W^{(1)}-W^{(3)}=&
b_0y
+(162b_0 + b_1)y^5
+(11186b_0 + 116b_1 + b_2)y^9
\\&
+(394498b_0 + 5081b_1 + 70b_2 + b_3)y^{13}
\\&
+(4628826b_0 + 65936b_1 + 1092b_2 + 24b_3 + b_4)y^{17}
\\&
+(-226397710b_0 - 2983519b_1 - 43758b_2- 781b_3 - 22b_4)y^{21}
\\&
+ \cdots.
\end{align*}
Since the shadow has minimum weight at least $5$,
we have $b_0=0$.
Hence, we have
\begin{align*}
W^{(1)}=&
\frac{a + b_1}{2} y^5
+\left(- 15 a + 58 b_1 + \frac{b_2}{2} + \frac{b}{2}\right)y^9
\\&
+\left( \frac{435 a}{2} + \frac{5081 b_1}{2} + 35 b_2  + \frac{b_3}{2}- 14 b
 - \frac{c}{2}\right)y^{13}
\\&
+\left(- 2030 a + 32968 b_1 + 546 b_2 + 12 b_3 + \frac{b_4}{2} + 189 b + 13 c
 + \frac{d}{2}\right)y^{17}
\\&
+ \left(\frac{27405a}{2} - \frac{2983519 b_1}{2}- 21879 b_2 -
 \frac{781b_3}{2}
- 11 b_4 - 1638 b \right.
\\&
\qquad 
\left. - \frac{325c}{2} - 12 d - \frac{e}{2}\right)y^{21}
+ \cdots.
\end{align*}
From the coefficients of $y^{13}$ and $y^{21}$,  
$e$ must be even.
The result follows.
\end{proof}

It is unknown whether there is a singly even self-dual
$[130,65,22]$ code or not  (see~\cite[Table~I]{HKWY}).

%
%

\section{Construction of singly even self-dual codes with minimal shadow}
\label{sec:const}

The following method for constructing singly even self-dual
codes was given in~\cite{Tsai92}.
Let $C$ be a doubly even self-dual code of length $8 t$.
Let $x$ be a vector of odd weight.
Let $C^0$ denote the subcode of $C$
consisting of all codewords which are orthogonal to $x$.
Then there are cosets
$C^1,C^2,C^3$ of $C^0$ such that ${C^0}^\perp = C^0 \cup C^1 \cup
C^2 \cup C^3$, where $C = C^0  \cup C^2$ and $x+C = C^1 \cup C^3$.
Then 
\begin{equation}\label{eq:code}
C(x)= (0,0,C^0) \cup  (1,1,C^2) \cup  (1,0,C^1) \cup  (0,1,C^3) 
\end{equation}
is a singly even self-dual code of length $8t+2$.
Using this method, a singly even self-dual code with minimal
shadow was constructed in~\cite{Tsai92} for the parameters
$[42,21,8]$ and $[58,29,10]$.
This may be generalized as follows.

\begin{thm}\label{thm:const}
Let $C$ be an extremal doubly even self-dual code of
length $8 t$ with covering radius $R$.
Then there is a vector $x$ of weight 
$2\lfloor{\frac {R+1}{2}}\rfloor-1$
such that
$C(x)$  in~\eqref{eq:code} is a singly even self-dual 
$[8t+2,4t+1,
\min\{
4\lfloor{\frac {t}{3}} \rfloor + 4,
2\lfloor{\frac {R+1}{2}}\rfloor\}
]$ code
with minimal shadow.
\end{thm}
\begin{proof}
Since there is a coset of minimum weight $R$,
there is a coset of minimum weight 
$2\lfloor{\frac {R+1}{2}}\rfloor-1$ (see~\cite[Fact 4]{A-P}).
We denote the coset by $x+C$, where
$x$ has weight $2\lfloor{\frac {R+1}{2}}\rfloor-1$.
Then the code $C(x)$ in~\eqref{eq:code}
is a self-dual code of length $8t+2$~\cite{Tsai92}.
The minimum weight of $C^0 \cup C^2$
is $4\lfloor{\frac {t}{3}} \rfloor + 4$.
The minimum weight of $C^1 \cup C^3$ is
$2\lfloor{\frac {R+1}{2}}\rfloor -1$.
Hence, $C(x)$ has minimum weight 
$\min\{
4\lfloor{\frac {t}{3}} \rfloor + 4,
2\lfloor{\frac {R+1}{2}}\rfloor\}$.

It remains to show that $C(x)$ has shadow of minimum weight $1$.
Without loss of generality, we may assume that $x \in C^1$.
Let $v$ be a vector of $C^1$.
Then $v$ is written as $x+c$, where $c \in C^0$.
Since $c \cdot x=0$, we obtain
\begin{equation}\label{eq:C1} 
\wt(x+c) \equiv \wt(x) \pmod 4. 
\end{equation}
Let $w$ be a vector of $C^3$.
Then $w$ is written as $x+c+c'$, where $c \in C^0$ and $c' \in C^2$.
From~\eqref{eq:C1} and $(x+c) \cdot c'=1$, we obtain
\[
\wt(x+c+c') \equiv \wt(x)+2 \pmod 4. 
\]
Suppose that
$\wt(x) \equiv 1 \pmod 4$ (resp.\ $\wt(x) \equiv 3 \pmod 4$).
Then $(0,0,C^0) \cup  (0,1,C^3)$ (resp.\ $(0,0,C^0) \cup  (1,0,C^1)$) 
is the doubly even subcode of $C(x)$.
In addition, the vector 
$(1,0,\ldots,0)$ (resp.\ $(0,1,0,\ldots,0)$)
is orthogonal to
any vector of the doubly even subcode.
This shows that the shadow has minimum weight $1$.
\end{proof}

%
%
%

We concentrate on 
singly even self-dual $[24k+10,12k+5,4k+2]$
codes with minimal shadow.
There is no extremal singly even
self-dual code of length $24k+10$
with minimal shadow for any nonnegative
integer $k$~\cite{BW}.
Hence, we have the following proposition.

\begin{prop}
If there is an extremal doubly even self-dual code of
length $24k+8$ with covering radius $R \ge 4k+1$,
then there is a 
singly even self-dual $[24k+10,12k+5,4k+2]$
codes with minimal shadow.
\end{prop}

The bordered double circulant extremal doubly even self-dual
$[80,40,16]$ code $B_{80,4}$ in~\cite{GH-DCC88} has
generator matrix
\[
\left(\begin{array}{ccccccccc}
{} & {} & {}      & {} & {} & 0     & 1  & \cdots &1  \\
{} & {} & {}      & {} & {} & 1     & {} & {}     &{} \\
{} & {} & I_{40}      & {} & {} &\vdots & {} & R     &{} \\
{} & {} & {}      & {} & {} & 1     & {} &{}      &{} \\
\end{array}\right),
\]
where 
$I_{40}$ is the identity matrix of order $40$ and 
$R$ is the $39 \times 39$ circulant matrix with
first row
\[
(111100000100101111101011101001101100011).
\]
It was shown in~\cite{HM-s} that $B_{80,4}$ has
covering radius $13$, where
a coset of minimum weight $13$ is given by
$x_{80} + B_{80,4}$ and $x_{80}$ has
the following support:
\[
\{ 2, 5, 8, 11, 14, 17, 20, 23, 26, 29, 32, 35, 38\}.
\]
We denote the code $B_{80,4}(x_{80})$ by $C_{82}$.

\begin{prop}\label{prop:82}
The code $C_{82}$ is a singly even self-dual $[82,41,14]$ code
with minimal shadow.
\end{prop}

For $k \ge 4$, only the extended quadratic residue code
$QR_{104}$ of length $104$
is the known extremal doubly even self-dual code of length $24k+8$.
It is not known whether $QR_{104}$ has covering radius
$R \ge 17$.
Our computer search failed to find a coset of weight $\ge 17$
in $QR_{104}$.

\section{New singly even self-dual $[82,41,14]$ codes}
\label{sec:82n}

In this section, 
we continue a search to find singly even self-dual
$[82,41,14]$ codes with weight enumerator for which no 
singly even self-dual code was previously known to exist.

Two self-dual codes $C$ and $C'$ of length $n$
are said to be {\em neighbors} if $\dim(C \cap C')=n/2-1$.
Any self-dual code of length $n$ can be reached
from any other by taking successive neighbors (see~\cite{C-S}).
By considering self-dual neighbors of $C_{82}$, 
we found $50$ singly even
self-dual $[82,41,14]$ codes $N_{82,i}$ $(i=1,2,\ldots,50)$ 
with weight enumerator for which no singly even
self-dual code was previously known to exist.
These codes are constructed as
\[
\langle (C_{82} \cap \langle x \rangle^\perp), x \rangle,
\]
where the supports $\supp(x)$ of $x$ are listed in Table~\ref{Tab:nei}.
The weight enumerators $W$ and the values $(\alpha, \beta)$
are also listed in the table.

\begin{table}[p]
\caption{Singly even self-dual $[82,41,14]$ neighbors $N_{82,i}$}
\label{Tab:nei}
\begin{center}
{\footnotesize
\begin{tabular}{c|l|c|c}
\noalign{\hrule height0.8pt}
Code & \multicolumn{1}{c|}{$\supp(x)$} & $W$ & $(\alpha,\beta)$\\
\hline
$N_{82, 1}$&$\{2,7,10,14,47,51,54,56,58,59,62,64,72,79\}$ &$W^C_{82,2}$&$(18,-750)$\\
$N_{82, 2}$&$\{2,7,12,13,14,42,47,56,57,59,61,71,73,79\}$ &$W^C_{82,3}$&$(1,-650)$ \\
$N_{82, 3}$&$\{6,9,11,42,44,47,51,56,59,61,75,77,78,79\}$ &$W^C_{82,3}$&$(1,-668)$ \\
$N_{82, 4}$&$\{5,9,45,49,55,59,61,63,66,70,71,72,75,81\}$ &$W^C_{82,3}$&$(1,-680)$ \\
$N_{82, 5}$&$\{2,3,9,13,14,39,40,47,49,56,57,64,77,82\}$ &$W^C_{82,3}$&$(1,-682)$ \\
$N_{82, 6}$&$\{2,3,9,11,12,45,46,49,53,64,72,75,77,80\}$&$W^C_{82,3}$&$(1,-686)$ \\
$N_{82, 7}$&$\{5,43,46,49,50,51,63,65,66,71,72,73,77,81\}$&$W^C_{82,3}$&$(1,-688)$ \\
$N_{82, 8}$&$\{3,4,5,7,9,45,48,55,56,58,61,66,73,77\}$&$W^C_{82,3}$&$(1,-692)$ \\
$N_{82, 9}$&$\{3,7,40,46,49,52,54,57,58,59,72,74,75,79\}$&$W^C_{82,3}$&$(1,-694)$ \\
$N_{82,10}$&$\{3,11,14,44,45,46,49,51,59,71,72,76,77,81\}$&$W^C_{82,3}$&$(1,-696)$ \\
$N_{82,11}$&$\{6,7,10,12,46,51,53,55,58,70,71,73,78,82\}$&$W^C_{82,3}$&$(1,-698)$ \\
$N_{82,12}$&$\{5,8,47,51,52,57,61,66,67,71,72,74,79,80\}$&$W^C_{82,3}$&$(1,-700)$ \\
$N_{82,13}$&$\{2,3,7,8,9,11,40,44,49,52,55,63,77,82\}$&$W^C_{82,3}$&$(1,-702)$ \\
$N_{82,14}$&$\{11,12,45,46,49,50,52,55,60,62,66,70,71,81\}$&$W^C_{82,3}$&$(1,-704)$ \\
$N_{82,15}$&$\{3,44,45,46,58,60,62,64,65,67,68,73,74,77\}$&$W^C_{82,3}$&$(1,-712)$ \\
$N_{82,16}$&$\{2,4,10,43,45,46,49,54,64,66,76,78,80,81\}$&$W^C_{82,3}$&$(1,-722)$ \\
$N_{82,17}$&$\{2,4,9,10,45,56,57,59,63,64,67,68,70,76\}$&$W^C_{82,3}$&$(1,-738)$ \\
$N_{82,18}$&$\{3,6,9,10,40,47,53,54,55,68,73,76,80,81\}$&$W^C_{82,3}$&$(1,-748)$ \\
$N_{82,19}$&$\{2,11,13,37,47,51,52,55,70,77,78,79,80,82\}$&$W^C_{82,3}$&$(2,-672)$ \\
$N_{82,20}$&$\{3,9,11,47,49,59,60,62,67,68,74,76,81,82\}$&$W^C_{82,3}$&$(2,-720)$ \\
$N_{82,21}$&$\{4,8,9,40,48,49,52,54,55,66,67,68,73,81\}$&$W^C_{82,3}$&$(2,-732)$ \\
$N_{82,22}$&$\{5,6,8,11,44,45,53,56,57,61,62,64,65,66\}$&$W^C_{82,3}$&$(2,-734)$ \\
$N_{82,23}$&$\{4,7,8,9,46,57,58,61,63,68,71,73,78,81\}$&$W^C_{82,3}$&$(0,-640)$ \\
$N_{82,24}$&$\{2,3,5,10,40,44,57,58,60,63,65,71,76,79\}$&$W^C_{82,3}$&$(0,-650)$ \\
$N_{82,25}$&$\{2,5,6,8,50,51,58,63,64,66,67,71,73,81\}$&$W^C_{82,3}$&$(0,-660)$ \\
$N_{82,26}$&$\{2,3,9,46,54,56,59,60,61,62,67,76,78,82\}$&$W^C_{82,3}$&$(0,-662)$ \\
$N_{82,27}$&$\{4,5,38,40,48,53,56,57,62,64,66,69,71,76\}$&$W^C_{82,3}$&$(0,-664)$ \\
$N_{82,28}$&$\{3,7,8,10,39,50,51,62,66,67,70,73,77,82\}$&$W^C_{82,3}$&$(0,-668)$ \\
$N_{82,29}$&$\{2,43,45,46,50,51,52,53,61,69,72,74,77,81\}$&$W^C_{82,3}$&$(0,-672)$ \\
$N_{82,30}$&$\{6,7,9,40,58,61,63,70,73,77,79,80,81,82\}$&$W^C_{82,3}$&$(0,-676)$ \\
$N_{82,31}$&$\{3,4,5,7,43,45,48,50,54,59,64,70,71,81\}$&$W^C_{82,3}$&$(0,-678)$ \\
$N_{82,32}$&$\{6,11,50,53,54,56,59,61,64,68,69,72,74,76\}$&$W^C_{82,3}$&$(0,-680)$ \\
$N_{82,33}$&$\{8,11,12,35,49,50,53,56,57,58,62,72,77,82\}$&$W^C_{82,3}$&$(0,-684)$ \\
$N_{82,34}$&$\{5,11,46,56,57,58,60,62,63,64,65,70,71,79\}$&$W^C_{82,3}$&$(0,-686)$ \\
$N_{82,35}$&$\{10,11,13,14,52,54,60,64,70,71,72,76,77,80\}$&$W^C_{82,3}$&$(0,-688)$ \\
$N_{82,36}$&$\{5,9,45,49,56,57,61,62,63,64,67,70,75,81\}$&$W^C_{82,3}$&$(0,-690)$ \\
$N_{82,37}$&$\{2,6,8,9,44,45,48,56,66,68,75,77,80,81\}$&$W^C_{82,3}$&$(0,-692)$ \\
$N_{82,38}$&$\{4,8,10,42,44,54,58,60,63,65,68,77,79,80\}$&$W^C_{82,3}$&$(0,-694)$ \\
$N_{82,39}$&$\{3,9,43,44,49,50,51,52,55,61,65,71,75,81\}$&$W^C_{82,3}$&$(0,-696)$ \\
$N_{82,40}$&$\{6,7,13,42,44,49,50,52,54,55,57,63,72,74\}$&$W^C_{82,3}$&$(0,-698)$ \\
\noalign{\hrule height0.8pt}
\end{tabular}
}
\end{center}
\end{table}

\setcounter{table}{0}
\begin{table}[thb]
\caption{Singly even self-dual $[82,41,14]$ neighbors $N_{82,i}$ (continued)}
\begin{center}
{\footnotesize
\begin{tabular}{c|l|c|c}
\noalign{\hrule height0.8pt}
Code & \multicolumn{1}{c|}{$\supp(x)$} & $W$ & $(\alpha,\beta)$\\
\hline
$N_{82,41}$&$\{2,4,8,13,45,46,49,51,58,65,66,73,74,80\}$&$W^C_{82,3}$&$(0,-700)$ \\
$N_{82,42}$&$\{3,9,12,45,54,55,59,64,66,72,74,75,78,80\}$&$W^C_{82,3}$&$(0,-706)$ \\
$N_{82,43}$&$\{2,4,9,10,45,55,56,57,60,64,67,69,72,74\}$&$W^C_{82,3}$&$(0,-708)$ \\
$N_{82,44}$&$\{4,9,11,40,45,46,55,57,63,64,65,71,72,74\}$&$W^C_{82,3}$&$(0,-710)$ \\
$N_{82,45}$&$\{3,44,45,46,57,60,61,62,63,70,71,74,75,77\}$&$W^C_{82,3}$&$(0,-712)$ \\
$N_{82,46}$&$\{7,40,44,45,52,53,55,56,67,68,71,76,79,81\}$&$W^C_{82,3}$&$(0,-716)$ \\
$N_{82,47}$&$\{3,5,9,12,42,45,47,51,53,55,60,64,68,75\}$&$W^C_{82,3}$&$(0,-718)$ \\
$N_{82,48}$&$\{6,39,44,45,54,60,62,64,65,75,77,78,79,81\}$&$W^C_{82,3}$&$(0,-720)$ \\
$N_{82,49}$&$\{2,5,9,43,60,61,62,64,68,71,74,76,80,81\}$&$W^C_{82,3}$&$(0,-724)$ \\
$N_{82,50}$&$\{3,7,9,13,43,46,48,49,50,52,58,60,63,81\}$&$W^C_{82,3}$&$(0,-728)$ \\
\noalign{\hrule height0.8pt}
\end{tabular}
}
\end{center}
\end{table}

Combined with the known result in~\cite{DGH},
the results in the previous section and this section
show the following:

\begin{prop}
There is a singly even self-dual $[82,41,14]$ code with shadow of
minimum weight $s$ for $s \in \{1,5, 9,13\}$.
\end{prop}

It remains to determine whether there is 
a singly even self-dual $[82,41,14]$ code with shadow of minimum weight $17$.

At the end of this section,
we summarize the current information on 
the weight enumerators which actually occur for the 
possible weight enumerators.
A singly even self-dual $[82,41,14]$ code with weight enumerator
$W^C_{82,1}$ is known (see Proposition~\ref{prop:82}).
A singly even self-dual $[82,41,14]$ code with weight enumerator
$W^C_{82,2}$ is known for $(\alpha,\beta)=(18,-750)$
(see Table~\ref{Tab:nei}).
A singly even self-dual $[82,41,14]$ code with weight enumerator
$W^C_{82,3}$ is known for 
\begin{align*}
\alpha=0  \text{ and }
\beta =& -640,-650,-656,-660,-662,-664,-668,-672,-676,\\
&-678,-680,-684,-686,-688,-690,-692,-694,-696, \\
&-698,\\
\alpha=1 \text{ and }
\beta=& -650,-668,-680,-682,-686,-688,-692,-694,-696,\\
& -698,-700,-702,-704,-712,-722,-738,-748, \\
\alpha=2  \text{ and }
 \beta =& -672,-720,-732,-734 
\end{align*}
(see \eqref{eq:82} and Table~\ref{Tab:nei}).

\bigskip
\noindent
{\bf Acknowledgment.}
This work was supported by JSPS KAKENHI Grant Number 15H03633.
The author would like to thank the anonymous referee for
useful comments.



\end{document}